\documentclass[12pt]{amsart}

\usepackage{amssymb}

\newtheorem{theorem}{Theorem}

\theoremstyle{statement}

\theoremstyle{definition}
\newtheorem{definition}[theorem]{Definition}

\newtheorem{example}[theorem]{Example}

\theoremstyle{corollary}
\newtheorem{corollary}[theorem]{Corollary}

\theoremstyle{conclusion}

\theoremstyle{proposition}
\newtheorem{proposition}[theorem]{Proposition}

\theoremstyle{property}

\theoremstyle{claim}

\theoremstyle{problem}

\theoremstyle{remark}

\theoremstyle{notation}

\begin{document}
	\title{Some remarks on selectively star-ccc spaces}
	\author{Yuan Sun$^{*}$}

	\address{College of Scienc, Beijing University of Civil Engineering and Architecture, Beijing 102616,China}

	\thanks{*Corresponding author. E-mail address: sunyuan@bucea.edu.cn}
	
	
	
	\maketitle

	\begin{abstract}
		In this paper, the author represent a unification and extension of
		concepts previously studied by several authors. By establish connections between the chain condition, selectively star-ccc properties and star-Lindelöf properties, the author provide that selectively $3$-star-ccc equals weakly strongly star-Lindelöf for normal spaces, which gives a negative answer to Question 4.10 raised by Xuan and Song
		in \cite{xuan2}, and use a different method from Xuan and Song to provide that the CCC implies selectively $2$-star-ccc, which gives a positive answer to Problem 4.4 in \cite{bal}. Finally, an example is given to distinguish these properties.
	\end{abstract}

	\section*{1.Introduction}
	
	A systematic study of classical selection principles was started by Scheepers in \cite{sch}. Later, Kočinac applied the star operator in the field of selection principles. He introduced and studied a number of star-selection principles. The survey papers \cite{ko} contain a detailed
	exposition on star selection principles theory. Aurichi introduced and studied a class of spaces which is a selective version of the countable chain condition in \cite{au}. Later, Bal and Kočinac in \cite{bal} defined and studied a star version of selectively ccc which constitute a more general spaces. Recently, Song and Xuan investigated the equivalence and properties of certain of these spaces. Several questions were raised in \cite{xuan2}. Most of these questions remains open.This paper arose in the attempt to analyse these questions.
	
	\section*{2. Notation and terminology}
	
	In this paper, we adopt the notation and terminology in \cite{van}; $\mathbb{N}^+$ denotes the set of strictly positive integers, $\omega$ denotes the first infinite ordinal and $\omega_1$ denotes the first uncountable ordinal. Note that unless otherwise specified, space means simply topological space, and no separation axioms are assumed.
	
	Recall the following definitions.
	
	\begin{definition} 
		A space $X$ has the \emph{countable chain condition} (CCC
		for short) if every pairwise disjoint open family of \(X\) is countable.
		Furthermore, a space \(X\) has the \emph{discrete countable chain
			condition} (DCCC) if every discrete open family of \(X\) is countable.
	\end{definition}
	
	\begin{definition} 
		A space $X$ is said to be \emph{weakly Lindelöf} if for
		every open cover $\mathcal{U}$ of $X$, there is some countable
		subset $\mathcal{V}$ of $\mathcal{U}$ such that $\overline{\bigcup \mathcal{V}}=X$.
	\end{definition}
	
	\begin{definition} (\cite{au}) 
		A space \(X\) is said to be a \emph{selectively ccc} space if for every sequence $\left(\mathcal{A}_n:n\in\omega\right)$ of maximal pairwise disjoint open families of $X$ there is a sequence $\left(A_n\in \mathcal{A}_n:n\in\omega\right)$ such that $\overline{\bigcup_{n\in\omega} A_n}=X$.
	\end{definition}
	
	Recall that if $B\subset X$ and $\mathcal{U}$ is a collection of subsets of $X$, then
	$\operatorname{st}^1\left(B,\mathcal{U}\right)=\bigcup\left\{U \in \mathcal{U}: U \cap B \neq \emptyset\right\}$. Inductively define $\operatorname{st}^{n+1}\left(B,\mathcal{U}\right)=\bigcup\left\{U \in \mathcal{U}: U \cap \operatorname{st}^n\left(B,\mathcal{U}\right) \neq \emptyset\right\}$
	or brevity we will write $\operatorname{st}\left(B,\mathcal{U}\right)$
	for $\operatorname{st}^1\left(B,\mathcal{U}\right)$ and $\operatorname{st}\left(x,\mathcal{U}\right)$ for $\operatorname{st}\left(\{x\},\mathcal{U}\right)$.
	
	\begin{definition} (\cite{bal})
		A topological space $X$ is said to be \emph{selectively $k$-star-ccc} if for every open cover
		$\mathcal{U}$ of $X$ and for every sequence $\left(\mathcal{A}_n:n\in\omega\right)$ of maximal pairwise disjoint open families of $X$, there is a sequence $\left(A_n\in \mathcal{A}_n:n\in\omega\right)$ such that $\operatorname{st}^k\left(\bigcup_{n\in\omega} A_n, \mathcal{U}\right)=X$. For brevity we will write selectively star-ccc for selectively $1$-star-ccc.
	\end{definition}
	
	Clearly, selectively ccc $\Longrightarrow$ selectively star-ccc and selectively $k$-star-ccc implies $\Longrightarrow$ selectively $k+1$-star-ccc for each $k\in \mathbb{N}^+$.
	
	\begin{definition} (\cite{van}) 
		A space $X$ is said to be \emph{$k$-star-Lindelöf} (respectively, \emph{strongly $k$-star-Lindelöf}) if for every open cover $\mathcal{U}$ of \(X\), there is some countable subset $\mathcal{V}$ of $\mathcal{U}$ (respectively, countable subset $B\subset X$) such that $\operatorname{st}^{k}(\bigcup \mathcal{V}, \mathcal{U})=X$ (respectively, $\operatorname{st}^{k}(B, \mathcal{U})=X$). We also write star-Lindelöf for $1$-star-Lindelöf and write strongly
		star-Lindelöf for strongly $1$-star-Lindelöf.
	\end{definition}
	
	Clearly, Lindelöf $\Longrightarrow$ strongly star-Lindelöf $\Longrightarrow$ star-Lindelöf. In general, strongly $k$-star-Lindelöf $\Longrightarrow k$-star-Lindelöf $\Longrightarrow$ strongly $k+1$-star-Lindelöf for each $k\in \mathbb{N}^+$.
	
	\begin{definition}
		A space $X$ is said to be \emph{weakly $k$-star-Lindelöf} (respectively, \emph{weakly strongly $k$-star-Lindelöf}) if for every open cover $\mathcal{U}$ of $X$, there is some countable subset $\mathcal{V}$ of $\mathcal{U}$ (respectively, countable subset $B\subset X$) such that $\overline{\operatorname{st}^k\left(\bigcup \mathcal{V}, \mathcal{U}\right)}=X$ (respectively,
		$\overline{\operatorname{st}^k\left(B, \mathcal{U}\right)}=X$). Write weakly star-Lindelöf for weakly $1$-star-Lindelöf and weakly strongly star-Lindelöf for weakly strongly $1$-star-Lindelöf.
	\end{definition}
	
	Clearly, Lindelöf $\Longrightarrow$ weakly Lindelöf $\Longrightarrow$ weakly strongly star-Lindelöf $\Longrightarrow$ weakly star-Lindelöf, and weakly strongly $k$-star-Lindelöf $\Longrightarrow$ weakly $k$-star-Lindelöf $\Longrightarrow$ weakly strongly $k+1$-star-Lindelöf for each $k\in \mathbb{N}^+$.
	
	In \cite{van}, van Douwen also consider the extension to $\omega$.
	
	\begin{definition} (\cite{van}) 
		A space $X$ is said to be $\omega$-star-Lindelöf (respectively, strongly $\omega$-star-Lindelöf) if for every open cover $\mathcal{U}$ of $X$, there is some $k\in\mathbb{N}^+$ and some countable subset $\mathcal{V}$ of $\mathcal{U}$ (respectively, countable subset $B\subset X$) such that $\operatorname{st}^k\left(\mathcal{V},\mathcal{U}\right)=X$ (respectively, $\operatorname{st}^{k}\left(B, \mathcal{U})=X\right)$.
	\end{definition}
	
	Note that using this style of argument, it is clear that $\omega$-star-Lindelöf equals strongly $\omega$-star-Lindelöf. Based on the notions above, we introduce the following definitions.
	
	\begin{definition}
		A space $X$ is said to be \emph{selectively $\omega$-star-ccc} if for every open cover $\mathcal{U}$ of $X$  and for every sequence $\left(\mathcal{A}_n: n\in\omega\right)$ of
		maximal pairwise disjoint open families of $X$, there is a sequence $\left(A_n\in \mathcal{A}_n:n\in\omega\right)$ and some $k\in\mathbb{N}^+$ such that $\operatorname{st}^k\left(\bigcup_{n\in\omega} A_n, \mathcal{U}\right)=X$.
	\end{definition}
	
	\begin{definition}
		A space $X$ is said to be \emph{weakly $\omega$-star-Lindelöf} (respectively, \emph{weakly strongly $\omega$-star-Lindelöf}) if for every open cover $\mathcal{U}$ of $X$, there is some $k\in\mathbb{N}^+$ and some countable subset $\mathcal{V}$ of $\mathcal{U}$ (respectively, countable subset $B\subset X$ such that $\overline{\operatorname{st}^k\left(\bigcup \mathcal{V}, \mathcal{U}\right)}=X$ (respectively, $\overline{\operatorname{st}^k\left(B, \mathcal{U}\right)}=X$).
	\end{definition}
	
	Clearly, the properties weakly $\omega$-SSL and weakly $\omega$-SL are equivalent.

	\section*{3. general implications}
	The results in this paper follow directly from work in \cite{van, xuan2}. For completeness, we present proofs in our current terminology. It is a well known fact in topology that for any open cover $\mathcal{U}$ of a CCC space $X$, there is a countable $\mathcal{V}\subset\mathcal{U}$ whose
	union is dense in $X$, i.e., every CCC space is weakly Lindelöf. By modifying the proof, we could obtain the following result.
	
	\begin{theorem}
		In a DCCC space $X$, for any open cover $\mathcal{U}$ of $X$ there is a countable $\mathcal{V}\subset\mathcal{U}$ such that $\operatorname{st}\left(\bigcup \mathcal{V},\mathcal{U}\right)$ is dense in $X$, i.e., every DCCC space is weakly star-Lindelöf.
	\end{theorem}
	
	\begin{proof}
		Suppose that \(X\) is a DCCC space which is not weakly
		star-Lindelöf. Let \(\mathcal{U}=\left\{U_{\alpha}: \alpha\in S\right\}\) be an open cover
		of \(X\) such that if \(\mathcal{V}\subset\mathcal{U}\) is countable, then \(\overline{\operatorname{st}\left(\bigcup \mathcal{V},\mathcal{U}\right)} \neq X\).For any \(\beta<\omega_1\), pick \(x_{\beta}\in X-\overline{\operatorname{st}\left(\bigcup \left\{U_{\alpha}: \alpha<\beta\right\},\mathcal{U}\right)}\) and let \(V_{\beta}=\operatorname{st}\left(x_{\beta},\mathcal{U}\right) \overline{\operatorname{st}\left(\bigcup \left\{U_{\alpha}: \alpha<\beta\right\},\mathcal{U}\right)}\). Clearly, \(V_{\beta}\) is open and
		\(V_{\beta}\cap V_{\beta^{'}}=\emptyset\) whenever \(\beta \neq \beta^{'}\).
		
		We claim that \(\mathcal{V}=\left\{V_{\beta}: \beta<\omega_1\right\}\) is discrete. Let \(y\in X\) and select \(U_y\in \mathcal{U}\) containing \(y\). If \(U_y\cap V_{\beta}\neq\emptyset\), then
		\(U_y\cap\operatorname{st}\left(x_{\beta},\mathcal{U}\right)\neq\emptyset\) which means that there exists \(U_\beta\in \mathcal{U}\) containing \(x_\beta\) such that \(U_y\cap U_{\beta}\neq\emptyset\), i.e., \(U_y\subset \operatorname{st}\left(\bigcup \left\{U_{\alpha}: \alpha\leqslant\beta\right\},\mathcal{U}\right)\). Then \(x_\gamma\notin U_y\) and hence \(U_y\cap V_{\gamma}=\emptyset\) for each \(\gamma>\beta\). On the other hand, by the definition,
		\(U_y\cap U_{\gamma}=\emptyset\) for each \(\gamma<\beta\). So for every \(y\in X\) there is an open \(U_y\in \mathcal{U}\) containing \(y\) which meets only one element of \(\mathcal{V}\). Hence claim.
		
		\(\mathcal{V}\) is an uncountable discrete collection of nonempty open sets. Therefore \(X\) is not DCCC that is a contradiction, so the result follows.
	\end{proof}
	
	By definitions of weakly star-Lindelöf and star-Lindelöf. It is easy to see that
	
	\begin{proposition} 
		Every weakly Lindelöf space is star-Lindelöf, and every weakly strongly star-Lindelöf space is strongly \(2\)-star-Lindelöf space. In general, Every weakly strongly \(k\)-star-Lindelöf space is strongly \(k+1\)-star-Lindelöf space and every weakly \(k\)-star-Lindelöf space is \(k+1\)-star-Lindelöf space for each \(k\in\mathbb{N}^+\).
	\end{proposition} 
	
	\begin{theorem}
		Every Lindelöf space is selectively star-ccc. In general, \(k\)-star-Lindelöf space is selectively \(k+1\)-star-ccc for each \(k\in\mathbb{N}^+\).
	\end{theorem}
	
	\begin{proof} 
		The implication Lindelöf space \(\Longrightarrow\) selectively star-ccc obtained by Bal in \cite{bal}. For the remaining implications, we only prove star-Lindelöf \(\Longrightarrow\)
		\(2\)-star-ccc. The proof of \(k\)-star-Lindelöf \(\Longrightarrow k+1\)-star-ccc for each \(k\in \mathbb{N}^+\) is similar.
		
		For any open cover \(\mathcal{U}\) of \(X\) and any sequence \(\left(\mathcal{A}_n:n\in \mathbb{N}^+\right)\) of maximal cellular open families of \(X\). Since \(X\) is star-Lindelöf space, there is some countable subset \(\mathcal{V}=\left\{U_n: n\in \mathbb{N}^+\right\}\) of \(\mathcal{U}\) such that \(\operatorname{st}\left(\bigcup \mathcal{V}, \mathcal{U}\right)=X\).
		For each \(n\in \mathbb{N}^+\), as \(\bigcup\mathcal{A}_n\) is dense in \(X\) there exists \(A_n\in\mathcal{A}_n\) such that \(U_n\cap A_n\neq\emptyset\). It follows that \(\bigcup\mathcal{V}\subset\operatorname{st}\left(\bigcup_{n\in\mathbb{N}^+} A_n, \mathcal{U}\right)\) and hence \(X=\operatorname{st}\left(\bigcup\mathcal{V}, \mathcal{U}\right)\subset\operatorname{st}^2\left(\bigcup_{n\in\mathbb{N}^+} A_n, \mathcal{U}\right)\), as required.
	\end{proof}
	
	In \cite{xuan2}, Xuan and Song showed that CCC implies selectively \(2\)-star-ccc, which gives a positive answer to Problem 4.4 in \cite{bal}. Now we can obtain their result by
	
	\begin{corollary} 
		Every CCC space is selectively \(2\)-star-ccc.
	\end{corollary}
	
	\begin{proof} 
		By Theorem 1, CCC \(\Longrightarrow\) weakly Lindelöf, by Theorem 2, weakly Lindelöf \(\Longrightarrow\) star-Lindelöf, and by Theorem 3, star-Lindelöf space \(\Longrightarrow\) selectively \(2\)-star-ccc. Therefor, we have the implication CCC \(\Longrightarrow\) selectively \(2\)-star-ccc.
	\end{proof}
	
	A similar proof shows that
	
	\begin{corollary} 
		Every DCCC space is selectively \(3\)-star-ccc.
	\end{corollary}
	
	For the converse, we have the following result.
	
	\begin{theorem}
		If \(X\) is regular weakly star-Lindelöf, then \(X\) is DCCC.
	\end{theorem}
	
	\begin{proof}
		Suppose \(X\) is not DCCC and \(\mathcal{U}=\left\{U_\alpha :\alpha\in S\right\}\) is a discrete
		collection of nonempty open sets. Pick \(x_\alpha\in U_\alpha\) arbitrarily and let \(D=\left\{x_\alpha:\alpha\in S\right\}\). Applying the regularity condition such that
		\(x_\alpha\in V_\alpha\subset\overline{V_\alpha}\subset U_\alpha\). Note that \(\left\{\overline{V_{\alpha}}:\alpha\in S\right\}\) is a discrete collection since \(\mathcal{U}\) is discrete and hence \(\mathcal{V}=\left\{U_\alpha: \alpha\in S\right\}\cup\left\{X-\overline{V_{\alpha}}\right\}\) is an open cover of \(X\). For each countable \(\mathcal{V}'\subset \mathcal{V}\) there exists \(U_\beta\notin \mathcal{V}'\), hence \(V_\alpha\cap\bigcup\mathcal{V}'=\emptyset\). Thus \(\mathcal{V}\) is an open cover witnesses that \(X\) is not weakly Lindelöf.
	\end{proof}
	
	In fact, Ikenaga in \cite{iken1,iken2} showed that for regular spaces, the DCCC equals weakly star-Lindelöf equals \(2\)-star-Lindelöf equals \(\omega\)-star-Lindelöf, and all the properties in between. Moreover, van Douwen in {[}3{]} showed that for normal spaces, the DCCC equals
	weakly strongly star-Lindelöf equals strongly \(2\)-star-Lindelöf and all the properties in between. For perfectly normal spaces, the DCCC equals the CCC equals weakly Lindelöf and equals star-Lindelöf and all the properties in between. Consequently, we obtain the following general
	implications.
	
	\begin{corollary} 
		For regular spaces. The properties of DCCC, selectively \(3\)-star-ccc, selectively \(\omega\)-star-ccc, weakly star-Lindelöf, weakly \(\omega\)-star-Lindelöf, weakly strongly \(2\)-star-Lindelöf, weakly strongly \(\omega\)-star-Lindelöf, \(2\)-star-Lindelöf, \(\omega\)-star-Lindelöf, strongly \(3\)-star-Lindelöf,\\ strongly \(\omega\)-star-Lindelöf and all the properties in between are equivalent.
	\end{corollary}
	
	\begin{corollary}
		For normal spaces. The properties of DCCC, selectively \(3\)-star-ccc, selectively \(\omega\)-star-ccc, weakly star-Lindelöf, weakly \(\omega\)-star-Lindelöf, weakly strongly star-Lindelöf, weakly strongly \(\omega\)-star-Lindelöf, \(2\)-star-Lindelöf, \(\omega\)-star-Lindelöf, strongly \(2\)-star-Lindelöf, strongly \(\omega\)-star-Lindelöf and all the properties in between are equivalent.
	\end{corollary}
	
	\begin{corollary}
		For perfectly normal spaces. The properties of CCC, selectively \(2\)-star-ccc, selectively \(\omega\)-star-ccc, weakly Lindelöf, weakly star-Lindelöf, weakly \(\omega\)-star-Lindelöf, weakly strongly star-Lindelöf, weakly strongly \(\omega\)-star-Lindelöf, star-Lindelöf, \(\omega\)-star Lindelöf, strongly \(2\)-star-Lindelöf, strongly \(\omega\)-star-Lindelöf and all the properties in between are equivalent.
	\end{corollary}
	
	By the Corollary $2$ above, we could give a negative answer to a Problem raised by Xuan and Song (\cite{xuan2}, Question 4.10); Since selectively \(3\)-star-ccc equals weakly star-Lindelöf for normal space, there does not exist a selectively \(3\)-star-ccc normal space which is not weakly star-Lindelöf. Note that weakly star-Lindelöf is called weakly star countable in \cite{xuan2}.

	\section*{4. counterexample}
	
	In this section, we use an example introduced by Xuan and Song to distinguish these properties.
	
	\begin{example} (\cite{xuan1}, Example 3.7)
		Let 
		$$X=\left(\left[0, \omega_1\right]\times\left[0, \omega_1\right)\right)\cup\left(I\left(\omega_1\right)\times\{\omega_1\}\right),$$ 
		where \(I(\omega_1)\) denotes the set of isolated ordinals in \(\left[0, \omega_1\right]\).
	\end{example}
	
	Xuan and Song showed that \(X\) is a Tychonoff DCCC space. Now we show the followings.
	
	{\bf Claim 1}. 
	\(X\) is a selectively \(3\)-star-ccc which is not
	selectively \(2\)-star-ccc.
	
	\begin{proof}
		The fact \(X\) is selectively \(3\)-star-ccc follows from Corollary 5. To show that \(X\) is not selectively \(2\)-star-ccc, let $$
		\mathcal{U}=\left\{\left[0, \omega_1\right]\times\left[0,\alpha\right]:\alpha<\omega_1\right\}\bigcup\left\{\left\{\gamma\right\}\times\left(\gamma, \omega_1\right]: \gamma\in I\left(\omega_1\right)\right\}, $$
		and for each \(n\in\omega\), let \(\mathcal{A}_n=\left\{\left\{x\right\}:x\in I\left(\omega_1\right)\times I\left(\omega_1\right)\right\}\). Note that \(\mathcal{U}\) is an open cover of \(X\) and \(\left(\mathcal{A}_n:n\in \omega\right)\) is a sequence of maximal pairwise disjoint families of \(X\).
		
		For each sequence \(\left(A_n\in\mathcal{A}_n: n\in \omega\right)\) there exists \(\gamma_0\in I\left(\omega_1\right)\) such that \(\operatorname{st}\left(\bigcup A_n, \mathcal{U}\right)\subset X-\left[\gamma_0, \omega_1\right]\times\left[\gamma_0, \omega_1\right]\). As \(\left\{\gamma_0\right\}\times\left(\gamma_0, \omega_1\right]\) is the unique element of \(\mathcal{U}\) containing point \(\left<\gamma_0, \omega_1\right>\) and \(\left\{\gamma_0\right\}\times\left(\gamma_0, \omega_1\right]\cap\operatorname{st}\left(\bigcup A_n, \mathcal{U}\right)=\emptyset\), hence \(\left<\gamma_0, \omega_1\right>\notin\operatorname{st}^2\left(\bigcup A_n, \mathcal{U}\right)\). It follows that \(X\) is not selectively \(2\)-star-ccc, as required.
	\end{proof}
	
	{\bf Claim 1}.
	\(X\) is a weakly star-Lindelöf which is not weakly strongly star-Lindelöf.
	
	\begin{proof} As \(X\) has the DCCC, it is weakly star-Lindelöf by Theorem 1. \(X\) is not weakly strongly star-Lindelöf follows from the fact that \(X\) is not even strongly \(2\)-star-Lindelöf, as we will shortly see in Claim 3.
	\end{proof}
	
	{\bf Claim 3}.
	\(X\) is a \(2\)-star-Lindelöf space which is not strongly \(2\)-star-Lindelöf.
	
	\begin{proof}
		Since \(X\) is DCCC, it is weakly star-Lindelöf and hence it is \(2\)-star-Lindelöf. Now we show that \(X\) is not strongly \(2\)-star-Lindelöf.
		
		Let \(\mathcal{U}\) be the open cover in Claim 1 above. For each countable subset \(B\subset X\), there exists \(\gamma_0\in I\left(\omega_1\right)\) such that \(\left<\gamma_0, \omega_1\right>\notin B\). As \(\left\{\gamma_0\right\}\times\left(\gamma_0, \omega_1\right]\) is the unique element of \(\mathcal{U}\) containing point \(\left<\gamma_0, \omega_1\right>\) and \(\left\{\gamma_0\right\}\times\left(\gamma_0, \omega_1\right]\cap\operatorname{st}\left(B, \mathcal{U}\right)=\emptyset\), hence \(\left<\gamma_0, \omega_1\right>\notin\operatorname{st}^2\left(B, \mathcal{U}\right)\). It follows that \(X\) is not strongly \(2\)-star-Lindelöf.
	\end{proof}
	
	Note that the above Claims show that \(X\) is a Tychonoff weakly star-Lindelöf (and hence selectively \(3\)-star-ccc space), which is not weakly strongly star-Lindelöf, is not selectively  \(2\)-star-ccc and is not strongly \(2\)-star-Lindelöf.

	{\bf Acknowledgement}. \ \ The authors would like to thank the
	referee for his (or her) valuable  remarks and suggestions which
	greatly improved the paper.

\end{document}